\documentclass[twoside]{article}

\usepackage{graphicx}
\usepackage{tikz}
\usepackage{amsfonts}
\usepackage{amsmath}
\usepackage{amssymb}
\usepackage[]{units}
\usepackage[export]{adjustbox}
\usepackage{setspace}
\usepackage{amsthm}
\usepackage{cleveref}
\usepackage{thmtools}
\usepackage{thm-restate}
\usepackage{enumitem}
\usepackage{titlesec}
\usepackage{fancyhdr}
\usepackage{abstract}

\declaretheorem[name=Theorem,numberwithin=section]{theorem}
\declaretheorem[name=Proposition, sibling=theorem]{proposition}
\newtheorem{lemma}[theorem]{Lemma}
\newtheorem{corollary}[theorem]{Corollary}

\theoremstyle{definition}
\newtheorem{definition}[theorem]{Definition}
\newtheorem{example}[theorem]{Example}

\usepackage[letterpaper, portrait, margin=1.5in]{geometry}
\titleformat{\section}
  {\normalfont\scshape\filcenter}{\thesection.}{1em}{}

\fancypagestyle{mypagestyle}{
  \fancyhf{}
  \fancyhead[OC]{PIECEWISE EXCLUDING GEODESIC LANGUAGES}
  \fancyhead[EC]{M. FRANKE}
  \fancyhead[R]{\thepage}
  
}
\title{\large\bf{PIECEWISE EXCLUDING GEODESIC LANGUAGES}}
\author{\normalsize{MARANDA FRANKE}}
\date{}

\renewcommand{\title}{\large}

\begin{document}
\maketitle
\vspace{-.5cm}
\begin{abstract}
\noindent \textsc{Abstract}. The complexity of a geodesic language has connections to algebraic properties of the group. Gilman, Hermiller, Holt, and Rees show that a finitely generated group is virtually free if and only if its geodesic language is locally excluding for some finite inverse-closed generating set. The existence of such a correspondence and the result of Hermiller, Holt, and Rees that finitely generated abelian groups have piecewise excluding geodesic language for all finite inverse-closed generating sets motivated our work. We show that a finitely generated group with piecewise excluding geodesic language need not be abelian and give a class of infinite non-abelian groups which have piecewise excluding geodesic languages for certain generating sets. The quaternion group is shown to be the only non-abelian 2-generator group with piecewise excluding geodesic language for all finite inverse-closed generating sets. We also show that there are virtually abelian groups with geodesic languages which are not piecewise excluding for any finite inverse-closed generating set.
\end{abstract}
\vspace{.3cm}
\section{\textsc{Introduction}}

For a group $G$ generated by a finite set $X$, Dehn's word problem asks if there exists an algorithm which determines whether or not a given word over $X\cup X^{-1}$ represents the trival element in $G$~\cite{Dehn}.  Dehn's word problem is known to be unsolvable in general~\cite{Boone}.  But for certain classes of groups, such as groups with a computable geodesic language for some generating set, there are solutions to the word problem. There are two known classes of groups with regular geodesic language for all finite generating sets: word hyperbolic groups~\cite{WP} and abelian groups ~\cite[Theorems 4.4 and 4.1]{N&S}. There are many known types of groups with regular geodesic language for some finite generating set: these include Coxeter groups~\cite{Howlett}, virtually abelian groups and geometically finite hyperbolic groups~\cite{N&S}, Artin groups of finite type and more generally Garside groups~\cite{Charney}, Artin groups of large type ~\cite{HR}, and groups hyperbolic relative to virtually abelian subgroups~\cite{AC}. The class of groups with regular geodesic language for some generating set is moreover closed under graph products~\cite{LMW}.

By considering more restrictive language classes than regular, it is possible to discover more properties of the underlying groups. In some cases, a characterization can be found. Gilman, Hermiller, Holt, and Rees show that a finitely generated group is virtually free if and only if its geodesic language is locally excluding for some finite symmetric (that is, inverse-closed) generating set ~\cite[Theorem 1]{GHHR}. Hermiller, Holt, and Rees show that a finitely generated group is free abelian if and only if, for some finite symmetric generating set, it has piecewise excluding geodesic language where the excluded piecewise subwords all have length one~\cite[Theorem 3.2]{LT}. Our research is motivated by the existence of these correspondences and by the following implications.

\begin{theorem}\label{PE} ~\cite[Proposition 6.2]{starfree}
Finitely generated abelian groups have piecewise excluding geodesic language for all finite symmetric generating sets.
\end{theorem}

\begin{theorem}\label{PT} ~\cite[Proposition 6.3]{starfree}
Finitely generated virtually abelian groups have piecewise testable geodesic language for some finite symmetric generating set.
\end{theorem}

Cannon gives an example showing that a finitely generated virtually abelian group can have a non-regular geodesic language for some finite symmetric generating set~\cite{N&S}. A natural question to investigate is if \Cref{PE} is a correspondence; that is, if groups with a piecewise excluding geodesic language for some generating set must be abelian. In Chapter 3, we show that a finitely generated group having piecewise excluding geodesic language does not imply that the group is abelian, even if the condition is strengthened to having piecewise excluding geodesic language for all finite symmetric generating sets.

\begin{restatable*}{proposition}{FbyA}
Let $K$ be a finitely generated abelian group, $H$ a finite group, and $G$ an extension of $H$ by $K$: $1\to H\to G\to K\to 1$. Then $G$ has a piecewise excluding geodesic language for some finite symmetric generating set.
\end{restatable*}

\begin{restatable*}{proposition}{Q}\label{Q8}
The quaternion group, $Q_8=<i,j,k\,|\,ijk^{-1},jki^{-1},kij^{-1},i^4>$, has piecewise excluding geodesic language for all finite symmetric generating sets.
\end{restatable*}

We show that the group $Q_8$ is a somewhat special 2-generator group and that the class of groups with piecewise excluding geodesic languages for all finite symmetric generating sets does not have nice closure properties.

\begin{restatable*}{theorem}{OnlyQ}\label{OQ}
The quaternion group, $Q_8$, is the only non-abelian 2-generator group with piecewise excluding geodesic language for all finite symmetric generating sets.
\end{restatable*}

\begin{restatable*}{proposition}{QxQ}
The class of groups which have piecewise excluding geodesic languages for all finite symmetric generating sets is not closed under direct products.
\end{restatable*}

Recall that \Cref{PT} shows that virtually abelian groups have piecewise testable geodesic language, a class which contains piecewise excluding geodesic languages. We show that the group property `virtually abelain' also does not correspond to piecewise excluding geodesic language by exhibiting a family of virtually abelian groups which have, for any finite symmetric generating set $A$, a geodesic word containing both a generator and its inverse.  By the proposition below, groups with a quotient isomorphic to a group in this family have, for any finite symmetric generating set, geodesic language which is not piecewise excluding.

\begin{restatable*}{corollary}{VA}
There are finitely generated virtually abelian groups whose geodesic language is not piecewise excluding  for any finite symmetric generating set.
\end{restatable*}

\begin{restatable*}{proposition}{Extns}
Let $G$ be an extension $1\to H\to G\stackrel{\pi}{\to}K\to 1$ of finitely generated groups $H$ and $K$ and let $A$ be a finite symmetric generating set for $G$.  If $awa^{-1}$ is geodesic in $K$ over the generating set $\pi (A)$ for some $a\in\pi(A)$ and $w\in\pi(A)^*$, then the geodesic language of $G$ over $A$ is not piecewise excluding.
\end{restatable*}

\section{\textsc{Background}}

In this paper, all groups we consider are finitely generated and all generating sets are finite and symmetric (that is, inverse closed).  We use the notation $[x,y]$ to mean the word $xyx^{-1}y^{-1}$. Let $G$ be a group with generating set $A$. We denote the identity element of $G$ by $1_G$ and use the notation $g=_G h$ to indicate that $g$ and $h$ are the same element of $G$. The smallest normal subgroup of $G$ containing a set $\{x_1,...,x_n\}$ is denoted by $<x_1,...,x_n>^N$. A set of \emph{normal forms} for $G$ over $A$ is a set $N$ of words over $A$ such that each element of $G$ has a unique representative in $N$. The \emph{Cayley graph} of $G$ over $A$, denoted $\Gamma (G,A)$, is the directed graph with a vertex labeled $g$ for each $g\in G$ and an edge labeled by $a$ from $g$ to $ga$ for each $a\in A$ and each $g\in G$. The graph is endowed with a metric by making each edge isometric to the unit interval and using the induced Euclidean metric. A \emph{geodesic word} in $\Gamma (G,A)$ is a word which labels a path of minimal length between two vertices in $\Gamma (G,A)$. The set of all finite length words over $A$, including the empty word, is denoted by $A^*$. A \emph{language} $L$ over $A$ is a subset of $A^*$.

\begin{definition} The \emph{geodesic language} of $G$ over $A$, denoted $\mathsf{Geo}(G,A)$, is the set of all geodesic words in $\Gamma (G,A)$.
\end{definition}

A language is \emph{regular} if it can be built out of finite subsets of the alphabet using the operations of concatenation, union, intersection, complementation, and * (Kleene closure); such an expression for a language is called a \emph{regular expression}. A language $L$ is regular if and only if $L$ can be recognized by a finite state automaton. For a reference on finite state automata and formal language theory, see~\cite{automata}.

\begin{example}\label{Cannon}
A virtually abelian group need not have regular geodesic language for every finite symmetric generating set. Cannon~\cite{N&S} exhibits the group $G=\mathbb{Z}^2 \rtimes \nicefrac{\mathbb{Z}}{2\mathbb{Z}}=<a,b,t\:|\;[a,b],t^2,tatb^{-1}>$, which has regular geodesic language with the generating set $\{a,b,t\}^{\pm 1}$ but not with the generating set $\{a,d,c,t\}^{\pm 1}$, where $c=a^2$ and $d=ab$.
\end{example}

The following three language class definitions can be found in~\cite{starfree}. A language $L$ over an alphabet $A$ is \emph{locally excluding} if there is a finite set of words $F\subset A^*$ such that $w\in L$ if and only if $w$ has no (contiguous) subword in $F$. A language $L$ over an alphabet $A$ is \emph{piecewise testable} if L is defined by a regular expression combining terms of the form $A^*a_1A^*a_2A^*\cdots A^*a_kA^*$, where $a_i\in A$, using the operations of concatenation, union, intersection, complementation, and * (Kleene closure). The string $a_1a_2\cdots a_n\in A^*$ is called a \emph{piecewise subword} of $w$ if $w=w_0a_1w_1a_2\cdots a_nw_n$ for some $w_i\in A^*$. A language $L$ over an alphabet $A$ is \emph{piecewise excluding} if there is a finite set of words $F\subset A^*$ such that $w\in L$ if and only if $w$ contains no piecewise subword in $F$.

\section{\textsc{Results}}

The following observation proved to be useful in showing particular geodesic languages were not piecewise excluding.
\begin{lemma}\label{Rmk} Let $G$ be a group generated by a finite symmetric generating set $A$. If $\mathsf{Geo}(G,A)$ is piecewise excluding, then $aa^{-1}$ must be an excluded piecewise subword for every $a\in A$ which does not represent the identity element of $G$.
\end{lemma}

\begin{proof} First note that if $a$ or $a^{-1}$ is excluded from $\mathsf{Geo}(G,A)$, then $a$ must represent the identity of the group. If $a\neq_G 1_G$, then $a,a^{-1}\in\mathsf{Geo}(G,A)$. For any $a\in A$, $aa^{-1}\notin \mathsf{Geo}(G,A)$ since $aa^{-1}=_G1_G$. In a piecewise excluding geodesic language, the only way to exclude the word $aa^{-1}$ from the language without excluding $a$ or $a^{-1}$ is by excluding $aa^{-1}$ as a piecewise subword.
\end{proof}

This suggests something strong about commutativity and seems to be evidence in favor of the existence of a correspondence between abelian groups and piecewise excluding geodesic lanugages. But there are non-abelian groups which have piecewise excluding geodesic language for some generating sets.

\begin{lemma} All finite groups have a generating set whose geodesic language is piecewise excluding.\end{lemma}

\begin{proof} Let $G$ be a finite group and let $A=G\setminus\{1_G\}$. Then $\mathsf{Geo}(G,A)$ is piecewise excluding. In particular, $\mathsf{Geo}(G,A)=A^*\setminus\{A^*aA^*bA^*\;|\;a,b\in A\}$, which is the set of all words over $A$ of length at most one.
\end{proof}

Finite groups may also have piecewise excluding geodesic language for smaller generating sets. Consider $D_8=<a,b,t\;|\;a^2,b^2,(ab)^4,ababt>$ and $A=\{a,b,t\}$. Then $\mathsf{Geo}(D_8,A)= A^*\setminus (\{A^*xA^*xA^*\;|\;x\in A\}\cup\{A^*xA^*yA^*zA^*\;|\;x,y,z\in A\})$, the set of all words over $A$ of length at most two which do not contain duplicate letters.  Note that with the generating set $B=\{a,b\}$, however, $D_8$ does not have piecewise excluding geodesic language, as $aba\in\mathsf{Geo}(D_8,B)$.

\FbyA

\begin{proof} Let the maps be $1\to H\stackrel{\iota}{\to} G\stackrel{\pi}{\to} K\to 1$ and let $A$ be a  finite symmetric generating set for $K$ with $1_K\notin A$. By ~\Cref{PE} $\mathsf{Geo}(K, A)$ is piecewise excluding; let $F$ be the finite set of excluded piecewise subwords. For each $a\in A$, choose a unique preimage under $\pi$, denoted $\bar{a}$, such that $\overline{a^{-1}}=\bar{a}^{-1}$. Let $\bar{A}=\{\bar{a}\;|\;a\in A\}$ and let $\bar{F}=\{\bar{a_1}\cdots\bar{a_n}\;|\;a_1\cdots a_n\in F\}$. Then words over $\bar{A}$ are geodesic if and only if they have no piecewise subword in $\bar{F}$. Let $\bar{H}=\iota(H\setminus\{1_H\})$. Note that as no generators in $\bar{H}$ represent the identity element of $G$, words of length one over $\bar{H}$ are geodesic; because each non-identity element of $H$ has a representative in $\bar{H}$, words of length two over $\bar{H}$ are not geodesic. Because $\iota(H)$ is a normal subgroup of $G$, for each $h\in H$ and each $a\in A$ there is an $h_a\in H$ such that $\bar{a}\iota(h)\bar{a}^{-1}=_G \iota(h_a)$. Suppose that $w\in (\bar{A}\cup \bar{H})^*$.  Write $w=a_1h_1a_2h_2\cdots a_nh_n$ where $a_i\in\bar{A}^*$ and $h_i\in\bar{H}$ for all $i\in\{1,...,n\}$. Then $w=_G \tilde{h}a_1a_2\cdots a_n$ where $\tilde{h}= (h_1)_{a_1}(h_2)_{a_1a_2}\cdots (h_n)_{a_1a_2\cdots a_n}$; that is, $w$ is equal in $G$ to a word in $(\bar{A}\cup\bar{H})^*$ with at most one element of $\bar{H}$ followed by $a_1a_2\cdots a_n$, the piecewise subword of $w$ over $\bar{A}$. Therefore any word in $(\bar{A}\cup\bar{H})^*$ with more than one letter from $\bar{H}$ or containing a piecewise subword over $\bar{A}$ which has a piecewise subword in $\bar{F}$ is not geodesic. Thus $\mathsf{Geo}(G,\bar{A}\cup\bar{H})$ is the piecewise excluding language whose set of excluded piecewise subwords is $\bar{H}^2\cup \bar{F}$.
\end{proof}

\Q

\begin{proof} Because the center of the group is $\{1_{Q_8},i^2\}$ and all other elements have order four, any set of elements of $Q_8$ containing at most one order four element (not including inverses) generates an abelian group. Hence any generating set for the non-abelian group $Q_8$ includes at least two order four elements which do not commute. Let $A$ be a finite symmetric generating set for $Q_8$, and let $a$ and $b$ be two order four elements in $A$ which do not commute. Then the eight words $1, a, a^{-1}, b, b^{-1}, a^2, ba$, and $b^{-1}a$ represent distinct elements of $Q_8$. The element $1$ has order one, the element $a^2$ has order two, and no two of the remaining (order four) elements can be equal because that would contradict that $a$ and $b$ do not commute and both have order four. Thus, any word of length at least three is not geodesic.  The language of geodesics is therefore piecewise excluding: the set of excluded piecewise subwords is the set of words of length three together with all words of length at most two that are not geodesic.
\end{proof}

The following two lemmas are used in the proof of \Cref{OQ}.

\begin{lemma}\label{Z5} All proper quotients of the group $G=\nicefrac{\mathbb{Z}}{5\mathbb{Z}}\rtimes\mathbb{Z}=<a,x\,|\,a^5,xax^{-1}a^2>$ are either abelian or, for some finite symmetric generating set, have a geodesic language which is not piecewise excluding.
\end{lemma}

\begin{proof} First note that a set of normal forms for $G$ over $A=\{a,x\}^{\pm 1}$ is $\{a^ix^n\,|\,i\in\{0,1,2,3,$ $4\},n\in\mathbb{Z}\}$. Observe that $xa=_G a^3x$, and so $x^na=_G a^{3^n}x^n$ for all $n\in\mathbb{Z}$, and that $x^4a=_G ax^4$. Any proper quotient $H$ of $G$ is isomorphic to $\nicefrac{G}{<a^{i_1}x^{n_1},a^{i_2}x^{n_2},...,a^{i_k}x^{n_k}>^N}$ for some $\{a^{i_j}x^{n_j}\}_{j=1}^k$ where for each $j\in\{1,...,k\}$, $i_j\in\{0,1,2,3,4\}$, $n_j\in\mathbb{Z}$, and $i_j,n_j$ are not both zero.

\noindent\emph{Case A: There is a $j\in\{1,...,k\}$ such that $n_j=0$.}

In this case $i_j\in\{1,2,3,4\}$, so $<a^{i_j}>\,=_G\,<a>$ is trivial in the quotient. Thus $H$ is a quotient of $\mathbb{Z}$, and so $H$ is abelian.

\noindent\emph{Case B: There is a $j\in\{1,...,k\}$ such that $i_j=0$ and $n_j\neq 0\;(\mathsf{mod} 4)$.}

In this case $a(x^{n_j})a^{-1}=_G ax^{n_j}a^4=_G a^{1+4\cdot3^{n_j}}x^{n_j}\in<x^{n_j}>^N,$
which implies that $a\in<x^{n_j}>^N$ for all possible $n_j$. Thus $H$ is a quotient of $\nicefrac{\mathbb{Z}}{n_j\mathbb{Z}}$, and so $H$ is abelian.

\noindent\emph{Case C: There is a $j\in\{1,...,k\}$ such that $i_j\neq 0$ and $n_j\neq 0\;(\mathsf{mod} 4)$.}

In this case $x^{n_j}(a^{i_j}x^{n_j})x^{-{n_j}}=_G x^{n_j}a^{i_j}\in<a^{i_j}x^{n_j}>^N$, which implies that $(x^{n_j}a^{i_j})^{-1}$ $=_G a^{-{i_j}}x^{-{n_j}}\in<a^{i_j}x^{n_j}>^N$. So $(a^{i_j}x^{n_j})(a^{-{i_j}}x^{-{n_j}})=_G a^{i_j+(5-i_j)3^{n_j}}\in<a^{i_j}x^{n_j}>^N.$ Note that $a^{i_j+(5-i_j)3^{n_j}}$ is a nontrivial element of $\nicefrac{\mathbb{Z}}{5\mathbb{Z}}$ for any $i_j\neq 0$ and $n_j\neq 0\;(\mathsf{mod}4 )$. Hence we have that $a\in<a^{i_j}x^{n_j}>^N$ in all subcases. Thus $H$ is a quotient of $\nicefrac{\mathbb{Z}}{{n_j}\mathbb{Z}}$, and so $H$ is abelian.

\noindent\emph{Case D: There is a $j\in\{1,...,k\}$ such that $i_j\neq 0$ and $n_j=0\;(\mathsf{mod} 4)$ is nonzero.}

Note that in this case $n_j-2=2\;(\mathsf{mod} 4)$, so $x^{{n_j}-2}a^{i_j}=_G a^{4i_j}x^{{n_j}-2}$. Thus
$x^{{n_j}-2}$ $(a^{i_j}x^{n_j})x^{-({n_j}-2)}=_G a^{-{i_j}}x^{n_j}\in<a^{i_j}x^{n_j}>^N$. This implies that $(a^{-{i_j}}x^{n_j})^{-1}=_G a^{i_j}x^{-{n_j}}\in$ $<a^{i_j}x^{n_j}>^N$. Hence $(a^{i_j}x^{n_j})(a^{i_j}x^{-{n_j}})=_G a^{2{i_j}}\in<a^{i_j}x^{n_j}>^N$. As $a^{2{i_j}}$ is a nontrivial element of $\nicefrac{\mathbb{Z}}{5\mathbb{Z}}$ for all $i_j\neq 0$,
in all subcases we have that $a\in<a^{i_j}x^{n_j}>^N$. Thus $H$ is a quotient of $\nicefrac{\mathbb{Z}}{n_j\mathbb{Z}}$, and so $H$ is abelian.

\noindent\emph{Case E: For every $j\in\{1,...,k\}$, $i_j=0$ and $n_j=0\;(\mathsf{mod} 4)$ is nonzero.}

Note that we can simplify the quotient to $\nicefrac{G}{<x^{\mathsf{gcd}(n_1,...,n_k)}>^N}$ in this case and that $\mathsf{gcd}(n_1,$ $...,n_k)\geq 4$. Consider the generating set $B=\{(ax),(xa)\}^{\pm 1}$. Observe that $(ax)^{-1}=_G a^3x^{-1}$ and $(a^3x^{-1})(xa)=_G a^{-1}$ so this is in fact a generating set for $G$, and thus for $H$ as well. Note that $(ax)(xa)(ax)^{-1}=_G a^3x$. As $a$ has order 5 in $H$, none of the generators in $B$ are equal in $H$ to $a^3x$. A word $w\in B^*$ of length two represents a group element $h\in G$ with a normal form $w'$ over $A$ that has an even power of $x$. So no words in $B^*$ of length two can be equal in $H$ to $a^3x$. Hence the word $(ax)(xa)(ax)^{-1}$ is geodesic in $H$. Therefore by \Cref{Rmk}, the geodesic language of $H$ over $B$ is not piecewise excluding.
\end{proof}

\begin{lemma}\label{BS} All proper quotients of the group $G=BS(1,2)=\,<a,t\,|\,tat^{-1}a^{-2}>$ are either abelian or, for some finite symmetric generating set, have a geodesic language which is not piecewise excluding.
\end{lemma}

\begin{proof} First note that a set of normal forms for $G$ over $\{a,t\}^{\pm 1}$ is $\{t^{-i}a^nt^j\,|\,i,j\in(\mathbb{N}\cup 0),n\in\mathbb{Z}, \text{and }2\nmid n \text{ if both } i,j>0\}$. Any proper quotient $H$ of $G$ is isomorphic to $\nicefrac{G}{<t^{-i_1}a^{n_1}t^{j_1},t^{-i_2}a^{n_2}t^{j_2},...,t^{-i_m}a^{n_m}t^{j_m}>^N}$ for some $\{t^{-i_k}a^{n_k}t^{j_k}\}_{k=1}^m$ where for each $k\in\{1,...,m$ $\}$, $i_k,j_k\in(\mathbb{N}\cup 0),n_k\in\mathbb{Z}$, $i_k,j_k,n_k$ are not all zero, and $i_k\neq j_k$ whenever $n_k=0$. Let $H$ be a non-abelian proper quotient of $G$. We first show that $H$ is a quotient of one of a specific collection of semi-direct products.

\noindent\emph{Case 1: Suppose there is an index $k$ such that $n_k\neq 0$.}

Note that $t^{-i_k}a^{n_k}t^{j_k}=_H 1_H$. If $i_k=j_k$ then $a^{n_k}=_H 1_H$. If $i_k\neq j_k$ then $t^{i_k-j_k}=_H a^{n_k}$, which is equal in $G$ to $t^{-1}a^{2n_k}t$. Hence $t^{i_k-j_k}=_H a^{2n_k}$, and so $a^{n_k}=_Ha^{2n_k}$, which implies that $a^{n_k}=_H 1_H$ and thus that $t^{i_k-j_k}=_H 1_H$.

If $n_k$ is even, then $ta^{n_k/2}=_G a^{n_k}t=_H t$, which implies that $a^{n_k/2}=_H 1_H$. So we can continue cutting the known order of $a$ in $H$ in half until we are left with an odd number, call it $n$. Then $ta=_Ga^2t$ implies the relations $ta^l=_H a^{2l}t$ for every integer $l$ such that $1\leq l<n/2$ and $ta^l=_H a^{2l-n}t$ for every integer $l$ such that $n/2<l<n$. These relations allow us to move $a^{\pm 1}$ past $t^{\pm 1}$ in either direction in any word.

Thus $H$ is a quotient of the semidirect product $\nicefrac{\mathbb{Z}}{n\mathbb{Z}}\rtimes\mathbb{Z}=<a,t\;|\;a^n,tat^{-1}a^{-2}>$ (if $i_k=j_k$) or of the semidirect product $\nicefrac{\mathbb{Z}}{n\mathbb{Z}}\rtimes\nicefrac{\mathbb{Z}}{|i_k-j_k|\mathbb{Z}}=<a,t\;|\;a^n,t^{|i_k-j_k|},tat^{-1}a^{-2}>$ (if $i_k\neq j_k$), with $n$ odd in either case.

\noindent\emph{Case 2: Suppose there is an index $k$ such that $n_k=0$.}

Then $i_k-j_k\neq 0$, so we may replace $t^{-1}$ with $t^{|i_k-j_k|-1}$ in the relation $tat^{-1}=_Ga^2$ and obtain the relations $t^{-1}a=_Ha^{2^{|i_k-j_k|-1}}t^{-1}$ and $at=_Hta^{1-2^{|i_k-j_k|-1}}$. So $a=_Gt^{-1}a^2t=_Ha^{2^{|i_k-j_k|}}$, which implies that $a^{2^{|i_k-j_k|}-1}=_H 1_H$. Hence $H$ is a quotient of the semidirect product $\nicefrac{\mathbb{Z}}{(2^{|i_k-j_k|}-1)\mathbb{Z}}\rtimes\nicefrac{\mathbb{Z}}{|i_k-j_k|\mathbb{Z}}=<a,t\;|\;a^{2^{|i_k-j_k|}-1},t^{|i_k-j_k|},tat^{-1}a^{-2}>$.

Because $H$ is non-abelian, Cases 1 and 2 show that $H$ is isomorphic to either a quotient of $\nicefrac{\mathbb{Z}}{n\mathbb{Z}}\rtimes\mathbb{Z}$ with $n$ odd and at least three or to a quotient of $\nicefrac{\mathbb{Z}}{n\mathbb{Z}}\rtimes\nicefrac{\mathbb{Z}}{|i-j|\mathbb{Z}}$ with $n$ odd  and at least three and $|i-j|\geq 2$.

Note that if $|i-j|=2$, then $a=_H t^2a=_H a^4t^2=_H a^4$. So in this case $a^3=_H 1_H$ and, moreover, H is isomorphic to a quotient of $S_3=<x,y\,|\,x^2,y^2,(xy)^3>$. This means that either $H$ has a geodesic language which is not piecewise excluding or $H$ is abelian: if $H\cong S_3$ then the word $xyx^{-1}$ is geodesic over the generating set $\{x,y\}^{\pm 1}$; if $H$ is a proper quotient of $S_3$ then $H$ is abelian. Note that if $|i-j|=3$, then $a=_H t^3a=_H a^8t^3=_H a^8$. So in this case $a^7=_H 1_H$ and, moreover, H is isomorphic to a quotient of $\nicefrac{\mathbb{Z}}{7\mathbb{Z}}\rtimes\nicefrac{\mathbb{Z}}{3\mathbb{Z}}=$ $<a,t\,|\,a^7,t^3,tat^{-1}a^{-2}>$. This means that either $H$ has a geodesic language which is not piecewise excluding or $H$ is abelian: if $H\cong\nicefrac{\mathbb{Z}}{7\mathbb{Z}}\rtimes\nicefrac{\mathbb{Z}}{3\mathbb{Z}}$ then the word $(at)t(at)^{-1}$ is geodesic over the generating set $\{(at),t\}^{\pm 1}$; if $H$ is a proper quotient of $\nicefrac{\mathbb{Z}}{7\mathbb{Z}}\rtimes\nicefrac{\mathbb{Z}}{3\mathbb{Z}}$ then $H$ has order 1,3, or 7 and so is abelian. 

What remains to be considered is when $H$ is isomorphic to either $\nicefrac{\mathbb{Z}}{n\mathbb{Z}}\rtimes\mathbb{Z}$ with $n$ odd and at least three or to $\nicefrac{\mathbb{Z}}{n\mathbb{Z}}\rtimes\nicefrac{\mathbb{Z}}{|i-j|\mathbb{Z}}$, with $n$ odd  and at least three and $|i-j|>3$. Let $B=\{(at),(ta)\}^{\pm 1}$. Note that $(at)^{-1}=_H t^{-1}a^{n-1}=_G (a^{\frac{n-1}{2}})t^{-1}$ and that $(a^{\frac{n-1}{2}}t^{-1})(ta)=_G a^{\frac{n+1}{2}}$. So as $(a^{\frac{n+1}{2}})^2=_H a$, the set $B$ is a generating set for $H$. Consider the word $(at)(ta)(at)^{-1}$, which is equal in $G$ to $a^3t$.

Note that $a^3t$ cannot be equal in $H$ to a single generator as $n\notin\{1,2\}$ and $|i-j|\neq 2$: the element $at$ is equal to in $H$ to $a^3t$ only if $a^2=_H 1$; the element $ta$ is equal in $H$ to $a^3t$ only if $a=_H 1$; the elements $(at)^{-1},\,(ta)^{-1}$ are equal  in $H$ to $a^3t$ only if $t^2=_H 1$. Note also that $a^3t$ cannot be equal in $H$ to a word of length two in the generators as $\;|i-j|\notin\{1,3\}$: the words $(at)(at),\,(at)(ta),\,(at)(ta)^{-1},\,(ta)(ta),\,(ta)(at),\,(ta)(at)^{-1},$ $(at)^{-1}(ta),\,(ta)^{-1}(at)$ are equal in $H$ to $a^3t$ only if $t=_H 1$; the words $(at)^{-1}(at)^{-1},$ $(at)^{-1}(ta)^{-1},\,(ta)^{-1}(ta)^{-1},\,(ta)^{-1}(at)^{-1}$ are equal in $H$ to $a^3t$ only if $t^3=_H 1$.

Therefore the word $(at)(ta)(at)^{-1}$ is geodesic over $B$. By \Cref{Rmk}, the geodesic language of $H$ over $B$ is not piecewise excluding.
\end{proof}

\OnlyQ

\begin{proof} Consider a minimal symmetric generating set $\{a,b\}^{\pm 1}$ for a two generator non-abelian group $G$ with piecewise excluding geodesic language for all finite symmetric generating sets. Because $G$ is non-abelian, $aba^{-1}\notin\{1,a,a^{-1},b\}$. Note that $aba^{-1}$ is not geodesic by \Cref{Rmk}, so it must then be equal in $G$ to either $b^{-1}$ or to a product of two generators. It can be shown, by straight-forward computations, that ten of the sixteen choices for words of length two over $\{a,b\}^{\pm 1}$ also lead to contradictions if they are equal in $G$ to $aba^{-1}$. For example, $aba^{-1}=_G ab^{-1}$ implies that $b^2=_G a$, which contradicts the assumption that $a $ and $b$ do not commute. Thus a representative of $aba^{-1}$ must be in the set $\{b^{-1},a^{-1}b,a^{-1}b^{-1},ba,b^2,b^{-1}a,b^{-2}\}$. Similarly, the possibilites for representatives of $bab^{-1}$ can be reduced to the set $\{a^{-1},b^{-1}a,b^{-1}a^{-1},ab,a^2,a^{-1}b,a^{-2}\}$.

Table 1 shows the group defined by only the two relations in each of the forty-nine possible pairs of choices for representaives of $aba^{-1}$ (along the first row) and for representatives of $bab^{-1}$ (along the first column). Note that by symmetry, the upper and lower diagonals are isomorphic groups. Most of the finite groups were found by entering the presentation into the GAP system and referencing the small group information within GAP; some (those listed below) required referencing groupprops.subwiki.org. We refer readers unfamiliar with GAP to~\cite{GAP}. The pairs $aba^{-1}=_G a^{-1}b^{-1}$ with $bab^{-1}=_G b^{-1}a^{-1}$ and $aba^{-1}=_G b^{-1}a$ with $bab^{-1}=_G a^{-1}b$ both returned the group [24:3], which is $SL_2(\nicefrac{\mathbb{Z}}{3\mathbb{Z}})$. The pair $aba^{-1}=_G b^{-2}$ with $bab^{-1}=_G a^{-2}$ returned the group [27,4], which is $\nicefrac{\mathbb{Z}}{9\mathbb{Z}}\rtimes\nicefrac{\mathbb{Z}}{3\mathbb{Z}}=$ $<x,y\,|\,x^9,y^3,yxy^{-1}x^{-4}>$. Notice that $aba^{-1}=a^{-1}b$ and $bab^{-1}=a^2$ are actually both the same relation, so this pair yields the group $BS(1,2)$. Groups which were reported to be infinite were calculated by hand using Tietze transformations.\\\vspace{-.25cm}
\begin{center}
\setlength\tabcolsep{3pt}\renewcommand{\arraystretch}{1.2}
    \begin{tabular}{| c | c | c | c | c | c | c | c |}
    \hline
    			&$b^{-1}$ 			&$a^{-1}b$		&$a^{-1}b^{-1}$		&$ba$			&$b^2$			&$b^{-1}a$		&$b^{-2}$\\\hline
    			
    $a^{-1}$ &{$Q_8^\ddagger$} & & & & & &\\\hline
    
    $b^{-1}a$ &{$\nicefrac{\mathbb{Z}}{3\mathbb{Z}}\rtimes\mathbb{Z}^\dagger$}  	&{$1$} & & & & &\\\hline
    
    $b^{-1}a^{-1}$ &{$\nicefrac{\mathbb{Z}}{6\mathbb{Z}}\times\nicefrac{\mathbb{Z}}{2\mathbb{Z}}$}	&{$1$}	&{$SL_2(\nicefrac{\mathbb{Z}}{3\mathbb{Z}})^\dagger$}	& & & &\\\hline
    
    $ab$ 		&{$\nicefrac{\mathbb{Z}}{3\mathbb{Z}}\rtimes\mathbb{Z}^\dagger$}	&{$1$}	&{$1$}	&{$1$}	& & &\\\hline
    
    $a^2$ 		&{$\nicefrac{\mathbb{Z}}{2\mathbb{Z}}$}	&{$BS(1,2)^\dagger$}	&{$S_3^\dagger$}	&{$\nicefrac{\mathbb{Z}}{3\mathbb{Z}}\rtimes\mathbb{Z}^\dagger$}	&{$1$} & &\\\hline
    
    $a^{-1}b$ 	&{$\nicefrac{\mathbb{Z}}{6\mathbb{Z}}\times\nicefrac{\mathbb{Z}}{2\mathbb{Z}}$}	&{$1$}	&{$\nicefrac{\mathbb{Z}}{5\mathbb{Z}}$}		&{$1$}	&{$S_3^\dagger$}	&{$SL_2(\nicefrac{\mathbb{Z}}{3\mathbb{Z}})^\dagger$}	&\\\hline
    
    $a^{-2}$ 	&{$\nicefrac{\mathbb{Z}}{6\mathbb{Z}}$}	&{$\mathbb{Z}$}	&{$\nicefrac{\mathbb{Z}}{6\mathbb{Z}}$}		&{$\nicefrac{\mathbb{Z}}{5\mathbb{Z}}\rtimes_\alpha\mathbb{Z}^\dagger$}	&{$\nicefrac{\mathbb{Z}}{3\mathbb{Z}}$}	&{$\nicefrac{\mathbb{Z}}{6\mathbb{Z}}$}	&{$\nicefrac{\mathbb{Z}}{9\mathbb{Z}}\rtimes_\beta\nicefrac{\mathbb{Z}}{3\mathbb{Z}}^\dagger$}\\\hline
    
\end{tabular}\\\vspace{.75em}
Table 1.\end{center}

The map $\alpha$ in the entry $\nicefrac{\mathbb{Z}}{5\mathbb{Z}}\rtimes_\alpha\mathbb{Z}$ is defined by the generator of $\mathbb{Z}$ conjugating the generator of $\nicefrac{\mathbb{Z}}{5\mathbb{Z}}$ to its square and the map $\beta$ in the entry $\nicefrac{\mathbb{Z}}{9\mathbb{Z}}\rtimes_\beta\nicefrac{\mathbb{Z}}{3\mathbb{Z}}$ is defined by the generator of $\nicefrac{\mathbb{Z}}{3\mathbb{Z}}$ conjugating the generator of $\nicefrac{\mathbb{Z}}{9\mathbb{Z}}$ to its fourth power. The group $\nicefrac{\mathbb{Z}}{3\mathbb{Z}}\rtimes\mathbb{Z}$ is the non-trivial semi-direct product.

The group $G$ must be a quotient of one of the groups in the table. All quotients of abelian groups are abelian, so $G$ cannot be a quotient of an abelian group in the table. The groups which are non-abelian but have a geodesic language which is not piecewise excluding for some finite symmetric generating set, demonstrated below, are denoted by a single dagger. We show below that all proper quotients of each of these groups are either abelian or have a geodesic language which is not piecewise excluding for some finite symmetric generating set. In each case of a geodesic language which is not piecewise excluding, a  check of all words of length at most two against a set of normal forms shows that the given length three word is geodesic.

The group $S_3=<a,b\,|\,a^2,b^2,(ab)^3>$ with the generating set $A=\{a,b\}^{\pm 1}$ has $aba^{-1}\in\mathsf{Geo}(S_3,A)$. So by \Cref{Rmk} $\mathsf{Geo}(S_3,A)$ is not piecewise excluding. The only proper quotients of $S_3$ are abelian. Thus $G$ cannot be a quotient of $S_3$.

The group $SL_2(\nicefrac{\mathbb{Z}}{3\mathbb{Z}})=<a,b\,|\,a^6,b^4,ab^{-1}ab^{-1}ab>$ with the generating set $A=\{a,b\}^{\pm 1}$ has $bab^{-1}\in\mathsf{Geo}(SL_2(\nicefrac{\mathbb{Z}}{3\mathbb{Z}}),A)$. So by \Cref{Rmk} $\mathsf{Geo}(SL_2(\nicefrac{\mathbb{Z}}{3\mathbb{Z}}),A)$ is not piecewise excluding. The only proper quotients of $SL_2(\nicefrac{\mathbb{Z}}{3\mathbb{Z}})$ are quotients of $A_4$ and quotients of $\nicefrac{\mathbb{Z}}{3\mathbb{Z}}$ (see groupprops.subwiki.org). The group $A_4=<a,b\,|\,a^3,b^2,(ab)^3>$ with the generating set $B=\{a,b\}^{\pm 1}$ has $bab^{-1}\in\mathsf{Geo}(A_4,B)$. By \Cref{Rmk} $\mathsf{Geo}(A_4,B)$ is not piecewise excluding. The only proper quotients of $A_4$ are abelian. Thus $G$ cannot be a quotient of $SL_2(\nicefrac{\mathbb{Z}}{3\mathbb{Z}})$.

The group $\nicefrac{\mathbb{Z}}{9\mathbb{Z}}\rtimes_\beta\nicefrac{\mathbb{Z}}{3\mathbb{Z}}=<x,y\,|\,x^9,y^3,yxy^{-1}x^{-4}>$ with the generating set $A=\{x,y\}^{\pm 1}$ has $xyx^{-1}\in\mathsf{Geo}(\nicefrac{\mathbb{Z}}{9\mathbb{Z}}\rtimes_\beta\nicefrac{\mathbb{Z}}{3\mathbb{Z}},A)$. By \Cref{Rmk} $\mathsf{Geo}(\nicefrac{\mathbb{Z}}{9\mathbb{Z}}\rtimes_\beta\nicefrac{\mathbb{Z}}{3\mathbb{Z}},A)$ is not piecewise excluding. As nontrivial proper subgroups of $\nicefrac{\mathbb{Z}}{9\mathbb{Z}}\rtimes\nicefrac{\mathbb{Z}}{3\mathbb{Z}}$ have order either $3$ or $3^2$, proper nontrivial quotients of $\nicefrac{\mathbb{Z}}{9\mathbb{Z}}\rtimes\nicefrac{\mathbb{Z}}{3\mathbb{Z}}$ have order either $3^2$ or $3$ and thus are abelian. Hence $G$ cannot be a quotient of $\nicefrac{\mathbb{Z}}{9\mathbb{Z}}\rtimes_\beta\nicefrac{\mathbb{Z}}{3\mathbb{Z}}$.

\Cref{BS} shows that the group $\nicefrac{\mathbb{Z}}{3\mathbb{Z}}\rtimes\mathbb{Z}=<a,x\,|\,a^3,xax^{-1}a>\cong\nicefrac{BS(1,2)}{<a^3>^N}$ and all its proper quotients are either abelian or have a finite symmetric generating set with geodesic language which is not piecewise excluding. Thus $G$ cannot be a quotient of $\nicefrac{\mathbb{Z}}{3\mathbb{Z}}\rtimes\mathbb{Z}$.

The group $\nicefrac{\mathbb{Z}}{5\mathbb{Z}}\rtimes_\alpha\mathbb{Z}=<a,x\,|\,a^5,xax^{-1}a^2>$ with the generating set $A=\{x,y\}^{\pm 1}$, where $y=_Gx^3a$, has $yxy^{-1}\in\mathsf{Geo}(\nicefrac{\mathbb{Z}}{5\mathbb{Z}}\rtimes_\alpha\mathbb{Z},A)$. By \Cref{Rmk} $\mathsf{Geo}(\nicefrac{\mathbb{Z}}{5\mathbb{Z}}\rtimes_\alpha\mathbb{Z},A)$ is not piecewise excluding. \Cref{Z5} shows that all proper quotients of $\nicefrac{\mathbb{Z}}{5\mathbb{Z}}\rtimes_\alpha\mathbb{Z}$ are either abelian or have a finite symmetric generating set with geodesic language which is not piecewise excluding. Thus $G$ cannot be a quotient of $\nicefrac{\mathbb{Z}}{5\mathbb{Z}}\rtimes_\alpha\mathbb{Z}$.

The group $BS(1,2)=<a,t\,|\,tat^{-1}a^{-2}>$ with the generating set $A=\{a,t\}^{\pm 1}$ has $t^{-1}at\in\mathsf{Geo}(BS(1,2),A)$. By \Cref{Rmk} $\mathsf{Geo}(BS(1,2),A)$ is not piecewise excluding. \Cref{BS} shows that all proper quotients of $BS(1,2)$ are either abelian or have a finite symmetric generating set with geodesic language which is not piecewise excluding. Thus $G$ cannot be a quotient of $BS(1,2)$.

The quaternion group, $Q_8$, denoted by a double dagger, has piecewise excluding geodesic language for all finite symmetric generating sets by \Cref{Q8}. The only proper quotients of $Q_8$ are abelian. Therefore, as all other possibilities lead to contradictions of our assumptions, the group $G$ must be isomorphic to $Q_8$.
\end{proof}

The class of groups with piecewise excluding geodesic languages for all finite symmetric generating sets does not even have one of the nicest closure properties one might hope for.

\QxQ

\begin{proof} Consider the generating set $A=\{i_1,j_1k_2,i_2,k_2\}^{\pm 1}$ for the group $G=Q_8\times Q_8$, where $i,j,k$ are as in the generating set for $Q_8=<i,j,k\,|\,ijk^{-1},jki^{-1},$ $kij^{-1},i^4>$ and the subscripts denote to which copy of $Q_8$ each belongs.  Consider the element $g=i_1(j_1k_2)i_1^{-1}=_G i_1^2j_1k_2$. Note that $g\notin A$. If $g=_G ab$ for some $a,b\in A$, then exactly one of $a,b$ must be $(j_1k_2)^{\pm 1}$ or $(k_2)^{\pm 1}$ and the other must be $i_1^{\pm 1}$ so that the projection into the second copy of $Q_8$ is $k_2$. But that forces the projection into the first copy of $Q_8$ to be one of $i_1^{\pm 1},k_1^{\pm 1}$. Hence $g$ cannot be written with fewer than three generators, and so $i_1(j_1k_2)i_1^{-1}\in \mathsf{Geo}(G,A)$. Thus $\mathsf{Geo}(G,A)$ is not piecewise excluding by \Cref{Rmk}.
\end{proof}

\begin{proposition} Let $\mathbb{Z}^n=<x_i,...x_n\,|\,[x_i,x_j]\text{ whenever }i\neq j>$ and let $G=\mathbb{Z}^n\rtimes_\phi\nicefrac {\mathbb{Z}}{2\mathbb{Z}}$ for some $n\in\mathbb{N}$ with either (1) $\phi(x_i)=x_i^{-1}$ for some $i\in\{1,...,n\}$ and $\phi(x_k)=x_k$ for all $k\in\{1,...,n\}\setminus\{i\}$ or (2) $\phi(x_i)=x_j$ for some $i,j\in\{1,...,n\}$ with $i\neq j$ and $\phi(x_k)=x_k$ for all $k\in\{1,...,n\}\setminus\{i,j\}$. Then for every finite symmetric generating set of $A$ of $G$, the geodesic language of $G$ over $A$ is not piecewise excluding. Moreover, there is a geodesic word over $A$ containing both a generator and its inverse.
\end{proposition}

\begin{proof}
Let $B=\{x_1,...,x_n,y\}^{\pm 1}$, where $\nicefrac {\mathbb{Z}}{2\mathbb{Z}}=<y>$, and let $N=\{x_1^{m_1}\cdots x_n^{m_n}y^\epsilon\;|\;m_i$ $\in\mathbb{Z}\text{ for all }i\in\{1,...,n\}\text{ and }\epsilon\in\{0,1\}\}$, a set of normal forms for $G$ over $B$. Let $A$ be any finite symmetric generating set for $G$. For every word $w\in A^*$, let $\rho_N (w)$ be the unique word in $N$ such that $\rho_N (w)=_G w$.\\

\noindent\underline{Case 1}: $\phi(x_i)=x_i^{-1}$ for some $i\in\{1,...,n\}$ and $\phi(x_h)=x_h$ for all $h\in\{1,...,n\}\setminus\{i\}$.

\noindent\emph{Subcase A: Suppose there is a generator $\alpha\in A$ such that $\rho_N (\alpha)=x_1^{m_1}\cdots x_i^m\cdots x_n^{m_n}$ for some $m\in\mathbb{Z}\setminus \{0\}$.}

Let $a\in A$ be the generator with $\rho_N (a)=x_1^{m_1}\cdots x_i^m\cdots x_n^{m_n}$ such that $m$ is maximal. Note that $m>0$ because $A$ is symmetric and that there must be at least one generator in $\beta\in A$ such that $\rho_N (\beta)=x_1^{k_1}\cdots x_n^{k_n}y$. Let $b\in A$ be the generator with $\rho_N (b)=x_1^{k_1}\cdots x_i^k\cdots x_n^{k_n}y$ such that $k$ is maximal. Suppose that $aba^{-1}$ is not geodesic. Observe that $\rho_N (aba^{-1})=x_1^{k_1}\cdots x_i^{2m+k}\cdots x_n^{k_n}y$.

\emph{subsubcase i: The word $aba^{-1}=_G\gamma$ for some $\gamma\in A$.} Because $m>0$ implies that $2m+k>k$, no generator $\gamma\in A$ with $\rho_N (\gamma)=x_1^{k_1}\cdots x_i^{2m+k}\cdots x_n^{k_n}y$ exists by maximality of $k$.

\emph{subsubcase ii: The word $aba^{-1}=_G\delta\zeta$ for some $\delta,\zeta\in A$ with $\rho_N (\delta)=x_1^{m'_1}\cdots x_i^p\cdots$ $x_n^{m'_n}$ and $\rho_N (\zeta)=x_1^{k'_1}\cdots x_i^q\cdots x_n^{k'_n}y$.} Then $p+q=2m+k$.  But the pair of inequalites $p\leq m$ and $q\leq k$ imply that $p+q\leq m+k<2m+k$. Thus no such pair of generators $\delta,\zeta\in A$ exists.

\emph{subsubcase iii: The word $aba^{-1}=_G\zeta\delta$ for some $\delta,\zeta\in A$ with $\rho_N (\delta)=x_1^{m'_1}\cdots x_i^p\cdots$ $x_n^{m'_n}$ and $\rho_N (\zeta)=x_1^{k'_1}\cdots x_i^q\cdots x_n^{k'_n}y$.} Then $q-p=2m+k$. But if $p\geq 0$, then $q-p\leq q\leq k< 2m+k$; if $p<0$, then $|p|\leq m$ implies that $q-p\leq k+m < 2m+k$. Thus no such pair of generators $\delta,\zeta\in A$ exists.

Hence $aba^{-1}$ is geodesic over $A$.

\noindent\emph{Subcase B: Suppose there is no generator $\alpha\in A$ such that $\rho_N (\alpha)=x_1^{m_1}\cdots x_i^m\cdots x_n^{m_n}$ for some $m\in\mathbb{Z}\setminus \{0\}$.}

Let $a\in A$ be the generator with $\rho_N (a)=x_1^{m_1}\cdots x_i^m\cdots x_n^{m_n}y$ such that $m$ is maximal. Let $b\in A$ be the generator with $\rho_N (b)=x_1^{k_1}\cdots x_i^k\cdots x_n^{k_n}y$ such that $k$ is minimal. Note that $k\neq m$, as otherwise $A$ generates only elements of $G$ with the power of $x_i$ in normal form either $0$ or $m$. Suppose that $aba^{-1}$ is not geodesic. Observe that $\rho_N (aba^{-1})=x_1^{k_1}\cdots$ $x_i^{2m-k}\cdots x_n^{k_n}y$. Let $\delta,\zeta\in A$ where $\rho_N(\delta)=x_1^{m'_1}\cdots x_i^p\cdots x_n^{m'_n}y^{\epsilon_1}$ and $\rho_N(\zeta)=x_1^{k'_1}\cdots x_i^q\cdots x_n^{k'_n}y^{\epsilon_2}$. Note that $\rho_N(\delta\zeta)=x_1^{m'_1+k'_1}\cdots x_i^{p+(-1)^{\epsilon_1}q}\cdots x_n^{m'_n+k'_n}y^{\epsilon_1+\epsilon_2 (\mathsf{mod}\,2)}$ and that if $\epsilon_1\neq 0$ [or $\epsilon_2\neq 0$], then $p=0$ [$q=0$]. So $p+(-1)^{\epsilon_1}q\leq m$ whenever $\epsilon_1+\epsilon_2 (\mathsf{mod}\,2)\neq 0$. If $aba^{-1}$ is equal in $G$ to a generator $\gamma\in A$ with $\rho_N (\gamma)=x_1^{k_1}\cdots x_i^{2m-k}\cdots x_n^{k_n}y$ or $aba^{-1}=_G\delta\zeta$, we have a contradiction to our choices of $a$ and $b$ since $k<m$ implies that $2m-k>m$. Hence $aba^{-1}$ is geodesic over $A$.\\

\noindent\underline{Case 2}: $\phi(x_i)=x_j$ for some $i,j\in\{1,...,n\}$ with $i\neq j$ and $\phi(x_h)=x_h$ for all $h\in\{1,...,n\}\setminus\{i,j\}$.

\noindent\emph{Subcase A: Suppose there is a generator $\alpha\in A$ such that $\rho_N (\alpha)=x_1^{m_1}\cdots x_i^m\cdots x_j^k\cdots$ $x_n^{m_n}$ for some $m\neq k\in\mathbb{Z}$.}

Let $c\in A$ be the generator with $\rho_N (c)=x_1^{m_1}\cdots x_i^{m_0}\cdots x_j^{k_0}\cdots x_n^{m_n}$ such that $|m_0-k_0|$ maximal. Note that $|m_0-k_0|>0$ by the assumption of this subcase and that there must be at least one generator $\beta\in A$ such that $\rho_N (\beta)=x_1^{k_1}\cdots x_n^{k_n}y$. Let $b\in A$ be the generator with $\rho_N (b)=x_1^{k_1}\cdots x_i^p\cdots x_j^q\cdots x_n^{k_n}y$ such that $|p-q|$ is maximal. If the signs of $m_0-k_0$ and $p-q$ agree or if $p=q$, let $a=c$; otherwise, let $a=c^{-1}$. Let $m,k\in\mathbb{Z}$ be such that $\rho_N(a)=_Gx_1^{m_1}\cdots x_i^m\cdots x_j^k\cdots x_n^{m_n}$. Suppose that $aba^{-1}$ is not geodesic. Observe that $\rho_N (aba^{-1})=x_1^{k_1}\cdots x_i^{m+p-k}\cdots x_j^{k+q-m}\cdots x_n^{k_n}y$. Becuase the signs of $m-k$ and $p-q$ do not disagree, the difference in the powers of $x_i$ and $x_j$ in $\rho_N (aba^{-1})$ is $|(m+p-k)-(k+q-m)|=|2(m-k)+(p-q)|=2|m-k|+|p-q|$.

\emph{subsubcase i: The word $aba^{-1}=_G\gamma$ for some $\gamma\in A$.} Because $|m-k|>0$ implies that $2|m-k|+|p-q|>|p-q|$, no generator $\gamma\in A$ with $\rho_N (\gamma)=x_1^{k_1}\cdots x_i^{m+p-k}\cdots x_j^{k+q-m}\cdots x_n^{k_n}y$ exists by our choice of $b$.

\emph{subsubcase ii: The word $aba^{-1}=_G\delta\zeta$ for some $\delta,\zeta\in A$ with $\rho_N (\delta)=x_1^{m'_1}\cdots x_i^l\cdots$ $x_j^r\cdots x_n^{m'_n}$ and $\rho_N (\zeta)=x_1^{k'_1}\cdots x_i^s\cdots x_j^t\cdots x_n^{k'_n}y$.} Then $l+s=m+p-k$ and $r+t=k+q-m$. But the pair of inequalities $|l-r|\leq|m-k|$ and $|s-t|\leq|p-q|$ imply that $|(l+s)-(r+t)|\leq |m-k|+|p-q|<2|m-k|+|p-q|$. Thus no such pair of generators $\delta,\zeta\in A$ exists.

\emph{subsubcase iii: The word $aba^{-1}=_G\zeta\delta$ for some $\delta,\zeta\in A$ with $\rho_N (\delta)=x_1^{m'_1}\cdots x_i^l\cdots$ $x_j^r\cdots x_n^{m'_n}$ and $\rho_N (\zeta)=x_1^{k'_1}\cdots x_i^s\cdots x_j^t\cdots x_n^{k'_n}y$.} Then $r+s=m+p-k$ and $l+t=k+q-m$. But the pair of inequalities $|l-r|\leq|m-k|$ and $|s-t|\leq|p-q|$ imply that $|(r+s)-(l+t)|\leq |m-k|+|p-q|<2|m-k|+|p-q|$. Thus no such pair of generators $\delta,\zeta\in A$ exists.

Hence $aba^{-1}$ is geodesic over $A$.

\noindent\emph{Subcase B: Suppose there is no generator $\alpha\in A$ such that $\rho_N (\alpha)=x_1^{m_1}\cdots x_i^m\cdots x_j^k\cdots$ $x_n^{m_n}$ for some $m\neq k\in\mathbb{Z}$.}

Let $a\in A$ be the generator with $\rho_N (a)=x_1^{m_1}\cdots x_i^m\cdots x_j^k\cdots x_n^{m_n}y$ such that $|m-k|$ is maximal. Let $b\in A$ be the generator with $\rho_N (b)=x_1^{k_1}\cdots x_i^p\cdots x_j^q\cdots x_n^{k_n}y$ such that $|p-q|$ is minimal. Note that $|m-k|\neq |p-q|$ as otherwise $A$ would only generate elements of $G$ with even differences in powers of $x_i$ and $x_j$ in normal forms without a $y$: the product $(x_1^{m'_1}\cdots x_i^c\cdots x_j^d\cdots x_n^{m'_n}y)(x_1^{k'_1}\cdots x_i^e\cdots x_j^f\cdots x_n^{k'_n}y)=_Gx_1^{m'_1+k'_1}\cdots x_i^{c+f}\cdots x_j^{d+e}\cdots x_n^{m'_n+k'_n}$; if $|c-d|=|e-f|$ then $|(c+f)-(d+e)|=|(c-d)-(e-f)|$, which is either $0$ or $2|c-d|$. Suppose that $aba^{-1}$ is not geodesic. Observe that $\rho_N (aba^{-1})=x_1^{k_1}\cdots x_i^{m+q-k}\cdots x_j^{k+p-m}\cdots x_n^{k_n}y$. Because $|m-k|>|q-p|$ implies that $(m-k)+(q-p)$ has the same sign as that of $m-k$, the difference in the powers of $x_i$ and $x_j$ in $\rho_N (aba^{-1})$ is $|(m+q-k)-(k+p-m)|=|(m-k)+(q-p)+(m-k)|=|(m-k)+(q-p)|+|m-k|>|m-k|$. The word $aba^{-1}$ must either be equal in $G$ to a generator $\gamma\in A$ with $\rho_N (\gamma)=x_1^{k_1}\cdots x_i^{m+q-k}\cdots x_j^{k+p-m}\cdots x_n^{k_n}y$ or to a product of generators $\delta,\zeta\in A$ with $\rho_N (\delta)=x_1^{l_1}\cdots x_i^t\cdots x_j^t\cdots x_n^{l_n}$ and $\rho_N (\zeta)=x_1^{l'_1}\cdots x_i^r\cdots x_j^s\cdots x_n^{l'_n}y$ such that $t+r=m+q-k$ and $t+s=k+p-m$ (note that $\delta\zeta=_G\zeta\delta$). By our choice of $a$, the largest possible difference in the powers of $x_i$ and $x_j$ in $\rho_N (\gamma)$, $\rho_N (\delta\zeta)$, or $\rho_N (\zeta\delta)$ is $|m-k|$. Therefore neither such a generator $\gamma$ nor such a pair of generators $\delta,\zeta$ exists. Hence $aba^{-1}$ is geodesic over $A$.

Thus $aba^{-1}\in\mathsf{Geo}(G,A)$ for the generators $a$ and $b$ defined in each subcase.  By \Cref{Rmk} $\mathsf{Geo}(G,A)$ cannot be piecewise excluding.
\end{proof}

\VA

\Extns

\begin{proof} Let $ab_1b_2\cdots b_na^{-1}\in \mathsf{Geo}(K,\pi (A))$ for some $a\in\pi(A),\;b_1,...,b_n\in\pi(A)$.  For each $x\in\pi(A)$, choose a unique preimage under $\pi$ in $A$, denoted $\bar{x}$, such that $\overline{x^{-1}}=\bar{x}^{-1}$. If $\bar{a}\bar{b_1}\bar{b_2}\cdots\bar{b_n}\bar{a}^{-1}\notin \mathsf{Geo}(G,A)$, then there is a word of length at most $n+1$ over $A$ equal in $G$ to $\bar{a}\bar{b_1}\bar{b_2}\cdots\bar{b_n}\bar{a}^{-1}$, say it is $x_1x_2\cdots x_k$. Then $\pi(x_1x_2\cdots x_k)=_K\pi(x_1)\pi(x_2)\cdots\pi(x_k)=_Kab_1b_2\cdots b_na^{-1}$, which implies that $ab_1b_2\cdots b_na^{-1}$ is not geo- desic, giving a contradiction.  Thus $\bar{a}\bar{b_1}\bar{b_2}\cdots\bar{b_n}\bar{a}^{-1}$ must be geodesic, and so by \Cref{Rmk}, $\mathsf{Geo}(G,A)$ cannot be piecewise excluding.
\end{proof}

\begin{corollary} Let $\mathbb{Z}^n=<x_i,...x_n\,|\,[x_i,x_j]\text{ whenever }i\neq j>$ and let $G=\mathbb{Z}^n\rtimes_\phi\nicefrac {\mathbb{Z}}{2\mathbb{Z}}$ for some $n\in\mathbb{N}$ with either $\phi(x_i)=x_i^{-1}$ for some $i\in\{1,...,n\}$ and $\phi(x_k)=x_k$ for all $k\in\{1,...,n\}\setminus\{i\}$ or  $\phi(x_i)=x_j$ for some $i,j\in\{1,...,n\}$ with $i\neq j$ and $\phi(x_k)=x_k$ for all $k\in\{1,...,n\}\setminus\{i,j\}$. Then any group with a quotient isomorphic to $G$ has a geodesic language which is not piecewise excluding for any finite symmetric generating set.
\end{corollary}

\bibliography{Bibliography2}{}

\begin{thebibliography}{10}

\bibitem{AC}
Yago Antol{{\'\i}}n and Laura Ciobanu.
\newblock Finite generating sets of relatively hyperbolic groups and
  applications to geodesic languages.
\newblock {\em Trans. Amer. Math. Soc.}, 368:7965--8010, 2016.

\bibitem{Boone}
W.W. Boone.
\newblock The word problem.
\newblock {\em Proc. Nat. Acad. Sci. U.S.A.}, 44:1061--1065, 1958.

\bibitem{Charney}
Ruth Charney and John Meier.
\newblock The language of geodesics for {G}arside groups.
\newblock {\em Math. Z.}, 248(3):495--509, 2004.

\bibitem{Dehn}
M.~Dehn.
\newblock {\"U}ber unedliche diskontinuierliche gruppen.
\newblock {\em Math. Ann.}, 71:116--144, 1911.

\bibitem{WP}
David~B.A. Epstein, J.W. Cannon, D.F. Holt, S.V.F. Levy, M.S. Paterson, and
  W.P. Thurston.
\newblock {\em Word processing in groups}.
\newblock Jones and Bartlett Publishers, 1992.

\bibitem{GAP}
The GAP~Group.
\newblock {\em {GAP -- Groups, Algorithms, and Programming, Version 4.8.7}},
  2017.

\bibitem{GHHR}
Robert~H. Gilman, S.~Hermiller, Derek~F. Holt, and Sarah Rees.
\newblock A characterisation of virtually free groups.
\newblock {\em Arch. Math.}, 89:289--295, 2007.

\bibitem{starfree}
S.~Hermiller, Derek~F. Holt, and Sarah Rees.
\newblock Star-free geodesic languages for groups.
\newblock {\em Internat. J. Algebra Comput.}, 17:329--345, 2007.

\bibitem{LT}
S.~Hermiller, Derek~F. Holt, and Sarah Rees.
\newblock Groups whose geodesics are locally testable.
\newblock {\em Internat. J. Algebra Comput.}, 18:911--923, 2008.

\bibitem{HR}
Derek~F Holt and Sarah Rees.
\newblock Artin groups of large type are shortlex automatic with regular
  geodesics.
\newblock {\em Proc. London Math. Soc.}, 104(3):486--512, 2010.

\bibitem{automata}
John~E. Hopcroft and Jeffrey~D. Ullman.
\newblock {\em Introduction to automata theory, languages, and computation}.
\newblock Addison-Wesley Publishing Company, 1979.

\bibitem{Howlett}
Robert~B. Howlett.
\newblock Miscellaneous facts about {C}oxeter groups, notes on lectures given
  at the {A}{N}{U} {G}roup {A}ctions {W}orkshop, {O}ctober 1993.
\newblock http://www.maths.usyd.edu.au/res/Algebra/How/anucox.html.

\bibitem{LMW}
J.~Loeffler, J.~Meier, and J.~Worthington.
\newblock Graph products and {C}annon pairs.
\newblock {\em Internat. J. Algebra Comput.}, 12:747--754, 2002.

\bibitem{N&S}
Walter~D. Neumann and Michael Shapiro.
\newblock Automatic structures, rational growth and geometrically finite
  hyperbolic groups.
\newblock {\em Invent. math.}, 120(1), 1994.

\end{thebibliography}
\bibliographystyle{plain}

\noindent Department of Mathematics, University of Nebraska, Lincoln, NE 68588, USA\\
E-mail address: marandafranke@gmail.com
\end{document}